\documentclass[12pt, twoside]{article}
\usepackage{amsmath,amsthm,amssymb}
\usepackage{times}
\usepackage{enumerate}
\usepackage{extarrows}
\usepackage{color}
\usepackage{tabularx}
\usepackage{multirow}

\usepackage {graphics}
\usepackage{graphicx}
\usepackage{url}
\def\titlerunning#1{\gdef\titrun{#1}}
\makeatletter
\def\author#1{\gdef\autrun{\def\and{\unskip, }#1}\gdef\@author{#1}}
\def\address#1{{\def\and{\\\hspace*{18pt}}\renewcommand{\thefootnote}{}%
		\footnote {#1}}%
	\markboth{\autrun}{\titrun}}
\makeatother
\def\email#1{e-mail: #1}
\def\subjclass#1{{\renewcommand{\thefootnote}{}%
		\footnote{\emph{Mathematics Subject Classification (2010):} #1}}}
\def\keywords#1{\par\medskip
	\noindent\textbf{Keywords.} #1}

\newtheorem{theorem}{Theorem}[section]
\newtheorem{lemma}[theorem]{Lemma}
\newtheorem{definition}[theorem]{Definition}
\newtheorem{proposition}[theorem]{Proposition}
\newtheorem{remark}[theorem]{Remark}

\newcommand{\R}{\mathbb{R}}
\newcommand{\ci}{\mathbb{S}}
\numberwithin{equation}{section}

\newcommand{\PreserveBackslash}[1]{\let\temp=\\#1\let\\=\temp}
\newcolumntype{C}[1]{>{\PreserveBackslash\centering}p{#1}}
\newcolumntype{R}[1]{>{\PreserveBackslash\raggedleft}p{#1}}
\newcolumntype{L}[1]{>{\PreserveBackslash\raggedright}p{#1}}

\newcolumntype{I}{!{\vrule width 1pt}}
\newlength\savedwidth

\begin{document}
	
	%%%%% To ease editing, add:
	
	\baselineskip=16pt
	
	%%%%%%%%%%%%%%%%
	
	%% In the running head, give an abbreviation of the title.
	\titlerunning{Time-periodic solutions of contact Hamilton-Jacobi equations}
	
	\title{Time-periodic solutions of contact Hamilton-Jacobi equations on the circle}
	
	\author{Kaizhi Wang \and Jun Yan \and Kai Zhao}

	\date{\today}
	
	\maketitle
	
	\address{Kaizhi Wang: School of Mathematical Sciences, Shanghai Jiao Tong University, Shanghai 200240, China; \email{kzwang@sjtu.edu.cn}
		\and Jun Yan: School of Mathematical Sciences, Fudan University, Shanghai 200433, China;
		\email{yanjun@fudan.edu.cn}
		\and
		Kai Zhao:  School of Mathematical Sciences, Fudan University, Shanghai 200433, China;
		\email{zhao$\_$kai@fudan.edu.cn}}
	\subjclass{37J50; 35F21; 35D40}

\begin{abstract}
	We are concerned with the existence and multiplicity of nontrivial time-periodic viscosity solutions to
\[
\partial_t w(x,t) + H( x,\partial_x w(x,t),w(x,t) )=0,\quad  (x,t)\in \mathbb{S} \times [0,+\infty). 
\]
We find that there are infinitely many nontrivial time-periodic viscosity solutions with different periods when $\frac{\partial H}{\partial u}(x,p,u)\leqslant-\delta<0$ by analyzing the asymptotic behavior of the dynamical system $(C(\ci ,\R),\{T_t\}_{t\geqslant 0})$, where $\{T_t\}_{t\geqslant 0}$ was introduced in \cite{WWY1}. Moreover, in view of the convergence of $T_{t_n}\varphi$, we get the existence of nontrivial periodic points of $T_t$, where $\varphi$ are  initial data satisfying certain properties. This is a long-time behavior result for the solution to the above equation with initial data $\varphi$.  At last, as an application, we describe to readers a bifurcation phenomenon for 
\[
\partial_t w(x,t) + H( x,\partial_x w(x,t),\lambda w(x,t) )=0,\quad  (x,t)\in \mathbb{S} \times [0,+\infty), 
\] 
when the sign of the parameter $\lambda$ varies. The structure of the unit circle $\ci$ plays an essential role here. The most important novelty is the discovery of the  nontrivial  recurrence of $(C(\ci ,\R),\{T_t\}_{t\geqslant 0})$.

\keywords{ Hamilton-Jacobi equations, nontrivial periodic solutions, bifurcation}
\end{abstract}

\tableofcontents

\section{Introduction and main results}
Throughout this paper, we use the symbol $M$ to denote an arbitrary smooth, connected, compact Riemannian manifold without boundary. Let $TM$ and $T^*M$ denote the tangent and cotangent bundle of $M$ respectively. Assume $H:T^*M\times\mathbb{R}\to \mathbb{R}$, $H=H(x,p,u)$, is a $C^\infty$ function satisfying:
\begin{itemize}
	\item [\bf(H1)] the Hessian $\frac{\partial^2 H}{\partial p^2} (x,p,u)$ is positive definite for each $(x,p,u)\in T^*M\times\R$;
	\item  [\bf(H2)] for each $(x,u)\in M\times\R$, $H(x,p,u)$ is superlinear in $p$;
	\item [\bf(H3)] there is a constant $\kappa >0 $  such that
	$$
	\Big| \frac{\partial H}{\partial u}(x, p,u)\Big| \leqslant \kappa ,\quad \forall (x,p,u)\in T^*M\times \R.
	$$
\end{itemize}
The associated Lagrangian is defined by
$$
L(x,\dot x,u):=\sup_{p\in T_x^*M} \{ \langle \dot x, p\rangle_x -H(x,p,u) \},\quad (x,\dot{x},u)\in TM\times\R.
$$
It is direct to check that $L$ satisfies
 \begin{itemize}
	\item [\bf(L1)] the Hessian $\frac{\partial^2 L}{\partial \dot{x}^2} (x,\dot{x},u)$ is positive definite for each $(x,\dot{x},u)\in TM\times\R$;
	\item  [\bf(L2)] for each $(x,u)\in M\times\R$, $L(x,\dot{x},u)$ is superlinear in $\dot{x}$;
	\item [\bf(L3)] there is a constant $\kappa >0 $  such that
	$$
	\Big| \frac{\partial L}{\partial u}(x, \dot{x},u)\Big| \leqslant \kappa ,\quad \forall (x,\dot{x},u)\in TM\times \R.
	$$
\end{itemize}
Define a family of nonlinear operators $\{T_t\}_{t\geqslant 0}$ from $C(M,\mathbb{R})$ to itself as follows. For each $\varphi\in C(M,\mathbb{R})$, denote by $(x,t)\mapsto T_t\varphi(x)$ the unique continuous function on $ (x,t)\in M\times[0,+\infty)$ such that
\begin{equation*}\label{eq:semigroup}
	T_t\varphi(x)=\inf_{\gamma}\left\{\varphi(\gamma(0))+\int_0^tL(\gamma(\tau),\dot{\gamma}(\tau),T_\tau\varphi(\gamma(\tau)))d\tau\right\},
\end{equation*}
where the infimum is taken among the absolutely continuous curves $\gamma:[0,t]\to M$ with $\gamma(t)=x$. It was  proved  in \cite{WWY1} that $\{T_t\}_{t\geqslant 0}$ is well-defined and is a one-parameter semigroup of operators.

Consider the dynamical system $(C(M,\R),\{T_t\}_{t\geqslant 0})$. The study of this dynamical system involves the space of continuous functions on $M$ ($C(M,\R)$), time ($t\geqslant 0$ parametrizes an irreversible continuous-time process), and the time-evolution law ($\{T_t\}_{t\geqslant 0}$). A  characteristic feature of dynamical theories is the emphasis on asymptotic behavior, especially in the presence of nontrivial recurrence, i.e., properties related to the behavior as time goes to infinity. This paper is devoted to the study of asymptotic behavior of the dynamical system $(C(M,\R),\{T_t\}_{t\geqslant 0})$. More precisely, we aim to search for nontrivial periodic points $\varphi$ of $T_t$, i.e., there is $\mathcal{T}_0>0$ such that $\varphi=T_{\mathcal{T}_0}\varphi$. Note that such a $\mathcal{T}_0> 0$ may be not unique. If the infimum of the set of all such $\mathcal{T}_0$'s is not 0, we say that $\varphi$ is a nontrivial periodic point of  $T_t$.

{\it We are interested in the nontrivial recurrence of the dynamical system $(C(M,\R),\{T_t\}_{t\geqslant 0})$. }

\subsection{What we have known and want to know}
Since the function  $(x,t)\mapsto T_t\varphi(x)$ is the unique viscosity solution (see \cite{CL,L} for definitions of viscosity solutions) of the evolutionary Hamilton-Jacobi equation
\begin{equation}\label{eq:HJe1}%\tag{HJ$_e$}
	\partial_t w(x,t) +  H( x,\partial_x w(x,t), w(x,t) )=0, \quad (x,t)\in M \times [0,+\infty)
\end{equation}
with $w(\cdot,0)=\varphi(\cdot)$ \cite{WWY1}, then one can deduce that 
\begin{itemize}
	\item $\varphi\in C(M,\R)$ is a common fixed point of $\{T_t\}_{t\geqslant 0}$ if and only if it is a viscosity solution of the ergodic Hamilton-Jacobi equation
	\begin{equation}\label{gen}
		H(x,Du(x),u(x))=0, \quad x\in M.
	\end{equation}
In this case, $\varphi$ is a trivial periodic point of $T_t$, or equivalently, it is a trivial periodic viscosity solution of \eqref{eq:HJe1}.
\item  $\varphi\in C(M,\R)$ is a nontrivial periodic point of $T_t$ if and only if the function  $(x,t)\mapsto T_t\varphi(x)$ is a nontrivial periodic viscosity solution of \eqref{eq:HJe1}.
\end{itemize}
Hence, nontrivial (resp. trivial) periodic points of $T_t$ and nontrivial (resp. trivial) time-periodic viscosity solutions of \eqref{eq:HJe1} are the same. We do not distinguish them in the following.

Let us recall what we have known about the asymptotic behavior of $T_t\varphi$: 
\begin{itemize}
	\item when $\frac{\partial H}{\partial u}=0$ ({\bf classical case}), there is a unique constant $c(H)\in\R$, called effective Hamiltonian \cite{LPV} or Ma\~n\'e's critical value \cite{Ma}, such that 
	\begin{align}\label{8-1}
		H(x,Du(x))=c(H)
	\end{align} 
has viscosity solutions. There are infinitely many viscosity solutions of \eqref{8-1}. For each $\psi\in C(M,\R)$, $T_t\psi+c(H)t$ goes to a viscosity solution of \eqref{8-1} as $t\to +\infty$ \cite{Fat-b}, which admits Lyapunov stability. Thus, there exists no nontrivial periodic viscosity solutions of 
\begin{align*}\label{8-2}
\partial_t w(x,t) +  H(x,\partial_x w(x,t))=c(H).
\end{align*}  
Moreover, each viscosity solution of \eqref{8-1} is Lyapunov asymptotically  stable.
\item  when $\frac{\partial H}{\partial u}\geqslant\delta>0$ ({\bf strict increasing} in $u$) , there is a unique fixed point $\varphi$ of $\{T_t\}_{t\geqslant 0}$ and for each $\psi\in C(M,\R)$, $T_t\psi$ goes to $\varphi$ as $t\to +\infty$. Furthermore, $\varphi$ admits Lyapunov stability. These are well-known results, see for example \cite{WWY2}.
\end{itemize}
For the above two cases, it is clear to see that the asymptotic behavior of $T_t$ is simple from the dynamical point of view. For the sake of brevity and readability, from now on we always assume that the Ma\~n\'e's critical value of $H(x,p,0)$ is 0.

We are concerned with 
\begin{itemize} 
	\item what will happen for more general  cases: will nontrivial periodic viscosity solutions of \eqref{eq:HJe1} appear? And how many? 
	
For example, assume $\frac{\partial H}{\partial u}\leqslant-\delta<0$ ({\bf strict decreasing} in $u$). 
	In this case, equation \eqref{gen} admits a unique forward weak KAM solution, denoted by $u_+$. See the Appendix for more details about weak KAM solutions. For this decreasing case, we have already known that \cite{WWY3}
	\begin{itemize}
		\item [(D1)]$T_t\varphi$ is bounded on $M \times[0,+\infty)$ if and only if 
		$\varphi \geq u_+$everywhere and 
		there exists $x_{0} \in M$ such that $\varphi\left(x_{0}\right)=u_+\left(x_{0}\right)$.
		In this case, for each $\varphi$ satisfying the above two conditions, there is $T_{\varphi}>0$ such that
		\[
		|T_t\varphi(x)|\leq K,\quad\forall  (x,t)\in M\times(T_{\varphi},+\infty)
		\]
		for some $K>0$ independent of $\varphi$.
		\item [(D2)]if there is $x_{0} \in M$ such that $\varphi\left(x_{0}\right)<u_+\left(x_{0}\right)$, then $\lim _{t \rightarrow+\infty} T_t\varphi=-\infty$ uniformly on $x \in M$.
		\item [(D3)] if $\varphi>u_{+}$ everywhere, then $\lim _{t \rightarrow+\infty} T_t\varphi=+\infty$ uniformly on $x \in M$.
	\end{itemize}
 \end{itemize}

\subsection{Main results}
%In \cite{WWY3} the nontrivial periodic solutions issue was not discussed. 
In this paper we aim to study the existence and multiplicity of nontrivial time-periodic viscosity solutions to
\begin{equation}\label{eq:HJe}\tag{E}
	\partial_t w(x,t) +  H( x,\partial_x w(x,t), w(x,t) )=0, \quad (x,t)\in \mathbb{S} \times [0,+\infty),
\end{equation}
where $\mathbb{S}$ is the unit circle, under (H1), (H2), and  the following assumption:
\begin{itemize}
	\item[\textbf{(H4)}] there are constants $\kappa>0$, $\delta >0 $  such that
	$$
	-\kappa\leqslant\frac{\partial H}{\partial u}(x, p,u) \leqslant -\delta <0 ,\quad \forall (x,p,u)\in T^*\mathbb{S}\times \R.
	$$
\end{itemize}

Let $u_+$ denote the unique forward weak KAM solution of equation
\begin{equation}\label{eq:HJs}\tag{S}
	H(x,u'(x), u(x))=0, \quad x\in \mathbb{S},
\end{equation}
and let $\Lambda_{u_+}:=\{ (x,(u_+)'(x),u_+(x)) : x \in \mathcal{D}(u_+)\}$, where $\mathcal{D}(u_+) $ denotes the set of points of differentiability of $u_+$. See the Appendix for definitions and properties of forward and backward weak KAM solutions of \eqref{eq:HJs}. Under (H1)-(H3), 
backward weak KAM solutions and viscosity solutions are the same.

The first main result of this paper is stated as follows.
\begin{theorem}\label{ma}
Assume (H1), (H2), (H4).  If, in addition,  
\begin{itemize}
	\item[$\mathrm{(A)}$] \quad \quad \quad 	$\frac{\partial H}{\partial p}\Big |_{\Lambda_{u_+}}\neq 0$ 
\end{itemize}
holds, then there exist infinitely many nontrivial time-periodic viscosity solutions with different periods of equation \eqref{eq:HJe}. 
\end{theorem}

\begin{remark} Some explanations on Theorem \ref{ma}:
	\begin{itemize}
		\item [(i)] We will show that equation \eqref{eq:HJs} has a unique backward KAM solution $u_-$ and  a unique forward KAM solution $u_+$. Moreover, $u_-=u_+=:u_0$, and $u_0$ is of class $C^\infty$. 
		 \item [(ii)] Note that $\Lambda_{u_0}$ (or equivalently, $\Lambda_{u_+}$) is a periodic orbit of the contact Hamiltonian flow $\Phi^H_t$ generated by
		 \begin{equation*}
		 	\left\{
		 	\begin{aligned}
		 		\dot x&=\frac{\partial H }{\partial p}(x,p,u),\\
		 		\dot p &=-\frac{\partial  H }{\partial x}(x,p,u)-\frac{\partial  H }{\partial u}(x,p,u) \cdot p,  \\ 
		 		\dot u&=\frac{\partial  H }{\partial p}(x,p,u) \cdot p- H(x,p,u).
		 	\end{aligned}
		 	\right.
		 \end{equation*}
	 Denote by $\mathcal{T}$ the period of the above periodic orbit throughout this paper. Note that the above contact Hamiltonian system is the characteristic equations of \eqref{eq:HJe}. This is why we call \eqref{eq:HJe} and \eqref{eq:HJs} contact Hamilton-Jacobi equations.  
	 \item [(iii)] We will also show that for each $n\in \mathbb{N}$, equation \eqref{eq:HJe} has infinite many nontrivial time  $\frac{\mathcal{T}}{n}$-periodic viscosity solutions.
	\end{itemize}
	
\end{remark}

As a consequence of Theorem \ref{ma}, we will discuss an interesting  bifurcation phenomenon for Hamilton-Jacobi equations
\begin{align}\label{9-1}\tag{E$_\lambda$}
	\partial_t w(x,t) +  H(x,\partial_x w(x,t), \lambda w(x,t))=0, \quad (x,t)\in \mathbb{S} \times [0,+\infty).
\end{align}
 The foundations of the one-parameter bifurcation theory has been laid by Poincar\'e who studied branching of solutions in the three-body problem. Applying Theorem \ref{ma} to \eqref{9-1}, we get the following result.
 
\begin{theorem}\label{ma1}
	Assume (H1), (H2), (H4).   
	\begin{itemize}
		\item [(1)] For $\lambda<0$, equation \eqref{9-1}  has a unique trivial time-periodic viscosity solution and has no nontrivial time-periodic viscosity solutions. In this case, the unique fixed point of $T_t$ is globally asymptotically stable in the Lyapunov sense.
		\item [(2)] For $\lambda=0$, equation \eqref{9-1} has infinitely many trivial time-periodic viscosity solutions and has no nontrivial time-periodic viscosity solutions. In this case, $T_t$ has no other $\omega$-limit points besides those fixed points.
		\item [(3)] If, In addition, assume
		\begin{itemize}
			\item[$\mathrm{(C)}$] \quad \quad 
		$
		\min_{p\in \R}H(x,p,0)<0,\quad \forall x\in\ci,
		$
		\end{itemize}
		then
	 there exists  $\lambda_0>0$ such that 
		for any $ \lambda\in (0,\lambda_0)$,  equation \eqref{9-1} has a unique trivial time-periodic smooth solution and infinitely many nontrivial time-periodic viscosity solutions. In this case, $T_t$ has a unique fixed point and infinitely many nontrivial  time-periodic points with different periods.
	\end{itemize}	
\end{theorem}

\begin{remark} Note that
	\begin{itemize}
		\item Items (1) and (2) are existing results. We will only prove the third one.
		\item Item (3) tells us that a new phenomenon has appeared: there are not just fixed points in the $\omega$-limit set of $T_t$. Nontrivial periodic points do exist. 
	\end{itemize}
\end{remark}

The last result of the present paper is a long-time behavior result for equation \eqref{eq:HJe}, which indicates the complexity of asymptotic behavior of the $(C(\ci ,\R),\{T_t\}_{t\geqslant 0})$.

\begin{theorem}\label{ma2} Assume (H1), (H2), (H4) and (A).
	Let $\varphi\in C(\mathbb{S},\R)$ satisfy 
	\begin{align}\label{7-30}
	\min_{x\in\mathbb{S}}(\varphi(x)-u_+(x))=0<\max_{x\in\mathbb{S}}(\varphi(x)-u_+(x)). 
	\end{align}
	Then the limit
	$$
	\lim_{n\to +\infty}T_{n\mathcal{T}+t}\varphi(x)=:w_{\varphi}(x,t), \quad \forall (x,t) \in \mathbb{S} \times [0,+\infty)
	$$
	exists and $w_{\varphi}(x,t)$ is a nontrivial  $\mathcal{T}$-periodic viscosity solution of \eqref{eq:HJe}.
\end{theorem}
\begin{remark}
	Note that condition \eqref{7-30} implies that $\varphi$ satisfies the conditions in (D1). 
\end{remark}

\subsection{Historical remarks}
The study of the viscosity solution theory for Hamilton-Jacobi equations has a long history dating back to the pioneering work of Crandall and Lions \cite{CL}. Since then, much progress has been made in this field. There are 
innumerable important works by many authors, see for instance \cite{Bar,I,T} and the references therein. Our methods and tools used in this paper come from \cite{WWY,WWY1,WWY2,WWY3}. 
In \cite{WWY} the authors established a variational principle for contact Hamiltonian systems under assumptions (H1)-(H3). This is the starting point of this series of works. Based on the variational principle, the authors of \cite{WWY1} dealt with the existence and representation formula of viscosity solutions to contact Hamilton-Jacobi equations under the same assumptions. The authors of \cite{WWY2} generalized part of Aubry-Mather-Ma\~n\'e-Fathi theory for classical Hamiltonian systems to contact Hamiltonian systems under (H1), (H2) and $0<\frac{\partial H}{\partial u}\leqslant \kappa$, and then they  \cite{WWY3} moved to strict decreasing case.

For Hamilton-Jacobi equations with time-periodic Hamiltonians $\bar{H}(x,p,t)$, there are many interesting works on the time-periodic viscosity solutions issue (see for instance, \cite{FM,BS,BR,WY,WY1}). To the best of our knowledge, there are few papers dealing with this subject for time-independent (contact) Hamilton-Jacobi equations.

%To show the existence of nontrivial time-periodic viscosity solutions of \eqref{eq:HJe} is an unusual issue, while the existence of such solutions of 
%evolutionary Hamilton-Jacobi equations with a time-periodic Hamiltonian $\bar H(t,x,p,u)$ is quite natural.

We will recall some definitions and preliminary results obtained in the aforementioned works on contact Hamiltonian systems in the Appendix. 
Especially, we will recall the concepts of the solution semigroups $T^-_t$ and $T^+_t$, which were introduced in \cite{WWY1}. We will use $T_t$ to denote $T^-_t$ throughout this paper except for the Appendix.

\section{Periodic solutions and bifurcation phenomenons }

\subsection{Proof of Theorem \ref{ma}}

We use $\mathcal{L}: T^*\ci\rightarrow T\ci$ to denote  the Legendre transform. Let
$\bar{\mathcal{L}}:=(\mathcal{L}, Id)$, where $Id$ denotes the identity map from $\R$ to $\R$. Then
\[
\bar{\mathcal{L}}:T^*\ci\times\R\to T\ci\times\R,\quad (x,p,u)\mapsto \left(x,\frac{\partial H}{\partial p}(x,p,u),u\right)
\]
is a diffeomorphism. Using $\bar{\mathcal{L}}$, we can define
the contact Lagrangian $L(x,\dot{x},u)$ associated to $H(x,p,u)$ as
\[
L(x,\dot{x},u):=\sup_{p\in T^*_x\ci}\{\langle \dot{x},p\rangle_x-H(x,p,u)\}.
\]
Then $L(x,\dot{x}, u)$ and $H(x,p,u)$ are Legendre transforms of each other, depending on conjugate variables $\dot{x}$ and $p$ respectively. Let $\Phi^H_t$ and $\Phi^L_t$ denote the local contact Hamiltonian flow and Lagrangian flow, respectively.

We divide the proof of Theorem \ref{ma} into several steps: from Lemma \ref{lem00} to Lemma \ref{leml}. In this section we always assume (H1), (H2), (H4) and (A). 

\begin{lemma}\label{lem00}
	 Equation \eqref{eq:HJs}
 has a unique backward weak KAM solution $u_-$ and a unique forward weak KAM solution $u_+$. Furthermore, $u_-=u_+=:u_0$ and $u_0$ is in fact a classical solution.
\end{lemma}

\begin{proof}
Equation \eqref{eq:HJs} has a unique forward weak KAM solution \cite{WWY2,WWY3}, denoted by $u_+$. Moreover, $u_-:=\lim_{t\to +\infty}T_tu_+$ is a backward weak KAM solution \cite{WY2021}.  Define the projected Aubry set by
	\[
	\mathcal{A}:=\{x\in\ci:\ u_-(x)=u_+(x)\}.
	\]
	In view of \cite{WWY2}, $\mathcal{A}$ is nonempty. Take an arbitrary point $x\in \mathcal{A}$. Denote by $x(t)$  the global $(u_-,L,0)$-calibrated curve passing through $x$. 
From (A),	the contact Hamiltonian flow $\Phi^H_t$ has no fixed points on $\Lambda_{u_+} $. So, we get that
	\[
	\mathcal{A}=\{x(t):t\in\R\}=\ci,\quad \tilde{\mathcal{A}}=\{(x(t),\dot{x}(t),u_-(x(t))):\ t\in\R\},
	\]
 where $\tilde{\mathcal{A}}$ denotes the Aubry set in $T\ci\times\R$ \cite{WWY2}. In this case, $u_-=u_+=:u_0$. Furthermore, from the graph property of the Aubry set, $u_0$ is the unique backward weak KAM solution. Recall that any backward weak KAM solution is $C^{1,1}$ on $\mathcal{A}$. The proof is complete.
\end{proof}

It is clear that $\tilde{\mathcal{A}}=\{(x(t),\dot{x}(t),u_0(x(t))): t\in\R\}$ is a periodic orbit of the contact Lagrangian flow $\Phi_t^L$.
Denote by $\mathcal{T}$ the period of this periodic orbit.
Let  
\begin{align}\label{9-199}
(x(t),p(t),u_0(x(t))):={\bar{\mathcal{L}}}^{-1}((x(t),\dot{x}(t),u_0(x(t))).
\end{align}

\begin{lemma}\label{lem0000}
Let $u_0$ be as in Lemma \ref{lem00}. Then $u_0$ is of class $C^\infty$. 
\end{lemma}

\begin{proof}
	Let $(x(t),p(t),u_0(x(t)))$ be as in \eqref{9-199}. By (A), it is direct to see that $\dot{x}(t)\neq 0$ for all $t\in\R$. 
	
	From
	\begin{equation}\label{9-200}
		\left\{
		\begin{aligned}
			\dot x&=\frac{\partial H }{\partial p}(x,p,u),\\
			\dot p &=-\frac{\partial  H }{\partial x}(x,p,u)-\frac{\partial  H }{\partial u}(x,p,u) \cdot p,  \\ 
			\dot u&=\frac{\partial  H }{\partial p}(x,p,u) \cdot p- H(x,p,u),
		\end{aligned}
		\right.
	\end{equation}
 we can get that
	\begin{equation}\label{9-201}
	\left\{
	\begin{aligned}
		\frac{dt}{dx}&=\big(\frac{\partial H }{\partial p}(x,p,u)\big)^{-1},\\
		\frac{dp}{dx} &=\big(\frac{\partial H }{\partial p}(x,p,u)\big)^{-1}\big(-\frac{\partial  H }{\partial x}(x,p,u)-\frac{\partial  H }{\partial u}(x,p,u) \cdot p\big),  \\ 
		\frac{du}{dx}&=\big(\frac{\partial H }{\partial p}(x,p,u)\big)^{-1}\big(\frac{\partial  H }{\partial p}(x,p,u) \cdot p- H(x,p,u)\big).
	\end{aligned}
	\right.
\end{equation}
 Since  $(x(t),p(t),u_0(x(t)))$ is a solution of \eqref{9-200}, then $t=t(x)$, $p(x):=p(t(x))$, $u_0(x)=u_0(t(x))$ is a solution of \eqref{9-201}, where $t=t(x)$ is uniquely determined by $x=x(t)$. Note that the vector field of \eqref{9-201} is of class $C^\infty$. So, $u_0(x)$ is of class $C^\infty$.
\end{proof}

\begin{lemma}\label{lem251} Let $u_0$ be as in Lemma \ref{lem00}. Then for any $x_0\in \ci$,  
	$$
	\liminf_{t\to +\infty} h_{x_0,u_0(x_0)}(x_0,t)< \limsup_{t\to +\infty}  h_{x_0,u_0(x_0)}(x_0,t).
	$$
\end{lemma}
\begin{proof}
	Let 
	$$
	B(x):= \frac{\partial H}{\partial p}(x, u'_0(x),u_0(x)),\quad x\in\ci.
	$$
	By (A) one can deduce that $B(x)\neq 0$ for all $x\in\ci$.
	Let $Z:=\int_0^1 -\big( B(\tau )\big)^{-1} d\tau\in \R \backslash \{0\}.$ Then it is clear that $|Z|=\mathcal{T}$. 
	For any fixed $x_0\in \ci$, we  define  
	\begin{equation}\label{eq:w}
		w (x,t):=u_0(x)+\epsilon+\epsilon \sin \Big(-\frac{\pi}{2}+  f(x)-f(x_0)+  \frac{2\pi}{Z}t\Big) ,\quad (x,t) \in \mathbb{S}\times[0,+\infty), 
	\end{equation}
	where
	$$
	f(x):=2\pi\cdot  \frac{\int_0^x \big( B(\tau )\big)^{-1} d\tau}{\int_0^1 \big( B(\tau )\big)^{-1} d\tau}= \frac{2\pi}{Z} \int_0^x -\big( B(\tau )\big)^{-1} d\tau>0,
	$$
	and $\epsilon>0$ is a parameter which will be determined later. 
	For any $\nu\in[0,1]$, let
	\begin{equation*}\label{eq:h_pp}
	 \widehat H^\nu _{pp}(x,t):=\int_0^1 s\int_0^1 \frac{\partial^2 H}{\partial p^2} \Big (x, u_0^\prime(x)+ \nu s \tau \cos\Big(-\frac{\pi}{x}+ f(x)-f(x_0)+ \frac{2\pi}{Z}t\Big)f^\prime(x),u_0(x)  \Big )  d\tau \ ds.  		
	\end{equation*}
	Let
	$$
	M_0:=\max_{\substack{ (x,t)\in \mathbb{S}\times [0,+\infty) \\
			\nu\in[0,1] }} \Big|\widehat H^\nu_{pp}(x,t) \Big |.
	$$
Take	
	\[
	\epsilon: =\min\Big\{\frac{1}{2}\frac{\delta Z^2}{4\pi^2 M_0} \cdot \min_{x\in \mathbb{S}} B^2(x),1\Big\},
	\]
where $\delta$ is as in (H4).
	
	We assert that 	
	$w(x,t)$ defined in \eqref{eq:w} is a subsolution of \eqref{eq:HJe} with $w(x_0,0)=u_0(x_0)$.
	It is clear that
		$$
		w(x_0,0)=u_0(x_0)+\epsilon+\epsilon \sin(-\frac{\pi}{2})=u_0(x_0).
		$$
	For any $x\in \mathbb{S}$, $t\geqslant 0$, let $F(x,t):= -\frac{\pi}{2}+ f(x)-f(x_0)+ \frac{2\pi}{Z}t$. In view of  $H \big(x, u_0^\prime (x) ,u_0(x) \big )=0 $ and $f^\prime(x)=-\frac{2\pi }{Z\cdot B(x)} $, by direct computation we have that
	\begin{align*}
		&\, \partial_t w (x,t)+H(x,\partial_x w (x,t), w (x,t))\\
		\leqslant &\, \epsilon \cos F(x,t)\cdot \frac{2\pi}{Z} +   H \Big(x, u_0^\prime (x)+\epsilon\cos F(x,t)\cdot f^\prime(x) ,u_0(x) \Big )- \delta \Big(\epsilon+\epsilon \sin F(x,t) \Big) \\
		\leqslant&\, \epsilon \cos F(x,t)\cdot \frac{2\pi}{Z}+   H \Big(x, u_0^\prime (x) ,u_0(x) \Big )+\epsilon\cos F(x,t) \cdot f^\prime(x)\cdot B(x) \\
		&\, +\epsilon^2 \cos^2 F(x,t) \cdot (f^\prime(x))^2  \cdot M_0- \delta \Big(\epsilon+\epsilon \sin F(x,t) \Big)\\
		\leqslant&\, \epsilon \cos F(x,t)\cdot \Big( \frac{2\pi}{Z}+f^\prime(x)\cdot B(x)  \Big)  +\epsilon^2 \cos^2 F(x,t) \cdot (f^\prime(x))^2  \cdot M_0- \delta \Big(\epsilon+\epsilon \sin F(x,t) \Big) \\
%		=&\, \epsilon^2 \cos^2 F(x,t) (f^\prime(x))^2  \cdot M_0- \delta \Big(\epsilon+\epsilon \sin F(x,t) \Big)\\		
		\leqslant &\,  \epsilon^2 \cos^2 F(x,t)\cdot (f^\prime(x))^2 M_0 - \delta \Big(\epsilon+\epsilon \sin  F(x,t) \Big) \\
		\leqslant &\, \epsilon \delta \Big( \frac{1}{2} \cos^2 F(x,t) -1 -\sin F(x,t)   \Big)
		= -\epsilon \delta \cdot \frac{1}{2} \Big(1 + \sin F(x,t) \Big)^2
		%= &\,-\epsilon \cdot \frac{1}{2} \Big[1 + \sin \Big( -\frac{\pi}{2}+  f(x)-f(x_0)+  \frac{2\pi}{T}t \Big)\Big]^2 
		\leqslant 0.
	\end{align*}
	So far, we have proved that $w (x,t)$ is a subsolution of \eqref{eq:HJe}.

	From the classical comparison principle (see, for instance, \cite{L}), one can deduce that
	$T_t w(x,0) \geqslant  w (x,t)$, since $T_t w(x,0)$ is a viscosity solution of \eqref{eq:HJe} while $w (x,t)$ is a viscosity subsolution of \eqref{eq:HJe}. Thus,  we get that
	$$
	h_{x_0,u_0(x_0)}(x,t)=h_{x_0,w(x_0,0)}(x,t)\geqslant T_t w(x,0) \geqslant  w (x,t), 
	$$
	which implies that  
	$$
	\limsup_{t\to +\infty} h_{x_0,u_0(x_0)}(x_0,t) > u_0(x_0)+\epsilon>u_0(x_0)=\liminf_{t\to +\infty} h_{x_0,u_0(x_0)}(x_0,t).
	$$	
	The proof is complete.
\end{proof}

\begin{lemma}\label{thm1} Let $u(t):=u_0(x(t))$. Then
	$$
	u(x,t):=\lim_{n\to +\infty}h_{x(0),u(0)}(x,n\mathcal{T}+t)
	$$
	is a nontrivial $\mathcal{T}$-periodic viscosity solution of equation \eqref{eq:HJe}.
\end{lemma}

\begin{proof}
	Since $(x(t),u(t))$ is $\mathcal{T}$-periodic and globally minimizing, then 
	\begin{align}\label{9-250}
	h_{x(0),u(0)} (x(0),\mathcal{T})=h_{x(0),u(0)} (x(\mathcal{T}),\mathcal{T})=u(\mathcal{T})=u(0).
	\end{align}
	Note that for each $x\in \ci$ and each $n\in\mathbb{N}$, 
\[	
h_{x(0),u(0)}(x, (n+1)\mathcal{T}) =  \inf_{y\in \ci}h_{y,h_{x(0),u(0)} (y,\mathcal{T})}(x,n\mathcal{T})\leqslant h_{x(0),h_{x(0),u(0)} (x(0),\mathcal{T})}(x,n\mathcal{T})
= h_{x(0),u(0)}(x, n\mathcal{T}). 
\]
Since $(x(t),u(t))$ is static, then for any $t\in \R$,
	$$
	u(t) \leqslant h_{x(0),u(0)}(x(t),s),\quad \forall s>0.
	$$
	Then for any $t>0$ and any $x\in \ci$, we have
	\begin{align*}
		u(0) \leqslant h_{x(0),u(0)}(x(0), t+\mathcal{T})=\inf_{y\in \ci}h_{y,h_{x(0),u(0)} (y,t)}(x(0),\mathcal{T})
		\leqslant   h_{x,h_{x(0),u(0)} (x,t)}(x(0),\mathcal{T}).
	\end{align*}
	So, we get that for any $t>0$ and any $x\in \ci$,
	$$
	h_{x(0),u(0)} (x,t) \geqslant h^{x(0),u(0)}(x,\mathcal{T}),
	$$
	which implies that $h_{x(0),u(0)} (x,t)$ is bounded from blow. 
	Thus one can define
	$$
	U(x):=\lim_{n\to \infty}h_{x(0),u(0)}(x, n\mathcal{T}).
	$$
	Let $u(x,t):=T_t U(x)$. Let $x_0:=x(0)$ and $v_0:=u(0)$. Then	
	\begin{align*}
		u(x,t)
		=  \inf_{y\in \ci} h_{y,U(y)}(x,t)
		=\lim_{n\to \infty}\inf_{y\in \ci} h_{y,h_{x_0,v_0}(y,n\mathcal{T}) }(x,t)
		=\lim_{n\to \infty} h_{x_0,v_0}(x,n\mathcal{T}+t)
	\end{align*}
	for any $(x,t)\in  \ci \times \R$.
	Notice that 
	\begin{align*}
		T_\mathcal{T} u(x,t)=&\, \inf_{y\in \ci} h_{y,u(y,t)}(x,\mathcal{T})\\
		=&\, \lim_{n\to \infty}\inf_{y\in \ci} h_{y,h_{x_0,v_0}(y,n\mathcal{T}+t) }(x,\mathcal{T})\\
		=&\, \lim_{n\to \infty} h_{x_0,v_0}\Big(x,(n+1)\mathcal{T}+t\Big)\\
		=&\, u(x,t).
	\end{align*}
Hence, $u(x,t)=T_\mathcal{T} u(x,t)=T_{\mathcal{T}+t}U(x)=u(x,t+\mathcal{T})$, for all $(x,t)\in \ci\times\R$.

Since $\liminf_{t\to +\infty} h_{x(0),u(0)}(x(0),t)\neq \limsup_{t\to +\infty} h_{x(0),u(0)}(x(0),t)$, then
 $u(x,t)$ is a nontrivial $\mathcal{T}$-periodic solution of equation \eqref{eq:HJe}.
\end{proof}

\begin{remark} Some explanations for the above lemma. 
	In fact, we can get a more general result stated as below: 
	 Assume $H:T^*M\times\mathbb{R}\to \mathbb{R}$, $H=H(x,p,u)$, is a $C^3$ function satisfying (H1)-(H3). If
			\begin{enumerate}
				\item[(i)] there exists a $\mathcal{T}'$-periodic static curve $(x(\cdot),u(\cdot)):\R \to M \times \R $ with $\mathcal{T}'>0$;
				\item[(ii)] $\liminf_{t\to +\infty} h_{x(0),u(0)}(x(0),t)<\limsup_{t\to +\infty} h_{x(0),u(0)}(x(0),t)$,
			\end{enumerate} 
			then 
			$$
			u(x,t):=\lim_{n\to +\infty}h_{x(0),u(0)}(x,n\mathcal{T}'+t)
			$$
			is a nontrivial $\mathcal{T}'$-periodic viscosity solution of equation 
			\begin{align*}
				\partial_tw(x,t)+H(x,\partial_xw(x,t),w(x,t))=0, \quad (x,t)\in M\times [0,+\infty).
			\end{align*}
		Here, $M$ is an arbitrary smooth, connected, compact Riemannian manifold without boundary. See the  Appendix for definitions of globally minimizing curves and static curves. The proof is quite similar to the one of Lemma \ref{thm1} and thus we omit it.
\end{remark}

\begin{lemma}\label{leml}
	If  equation \eqref{eq:HJe} has a nontrivial $\mathcal{T}$-periodic viscosity solution, then for each $n\in \mathbb{N}$, equation \eqref{eq:HJe} has infinitely many nontrivial $\frac{\mathcal{T}}{n}$-periodic viscosity solutions.
\end{lemma}
\begin{proof}
	If $w(x,t)$ is a nontrivial $\mathcal{T}$-periodic viscosity solution of \eqref{eq:HJe}, then for any $a\in [0,\mathcal{T}]$, $w(x,t+a)$ is also a nontrivial $\mathcal{T}$-periodic viscosity solution of \eqref{eq:HJe}. For each given $n\in \mathbb{N}$, define 
	$$
	v_n(x,t):=\min \Big\{w(x,t), w(x,t+\frac{\mathcal{T}}{n} ), w(x,t+\frac{2\mathcal{T}}{n}),\cdots,w(x,t+\frac{(n-1)\mathcal{T}}{n})   \Big\}
	$$ 
	which is a nontrivial $\frac{\mathcal{T}}{n}$-periodic viscosity solution of \eqref{eq:HJe}.
	
	 Hence for any $a\in [0,\frac{\mathcal{T}}{n} ]$, $v_n(x,t+a)$ is also a nontrivial $\frac{\mathcal{T}}{n}$-periodic viscosity  solution of \eqref{eq:HJe} and for any   $a_1$, $a_2 \in [0,\frac{\mathcal{T}}{n}]$, 
	$$
	\widetilde v_n(x,t):=\min\{v_n(x,t+a_1), v_n(x,t+a_2)  \}
	$$
	is also a $\frac{\mathcal{T}}{n}$-periodic viscosity solution of \eqref{eq:HJe}. Therefore,  equation \eqref{eq:HJe} has infinitely many nontrivial $\frac{\mathcal{T}}{n}$-periodic viscosity solutions.
\end{proof}

Theorem \ref{ma} is a direct consequence of Lemmas \ref{lem00}, \ref{lem0000}, \ref{lem251}, \ref{thm1},  \ref{leml}.

\subsection{Proof of Theorem \ref{ma1}}
In order to prove Theorem \ref{ma1}, we only need to show item (3) since items (1) and (2) are known results.

\begin{lemma}\label{lem:7}
	Under the assumptions imposed on $H$ in item (3) in Theorem \ref{ma1}, there exists $\lambda_0>0$ such that  the contact vector field generated by $H(x,p,\lambda u)$ has no fixed points on $\Lambda_{u_\lambda ^+}$ for $\lambda\in  (0,\lambda_0)$, where $u_\lambda ^+$ denotes the unique forward weak KAM solution of $H(x,u'(x),\lambda u(x))=0$.
\end{lemma}

\begin{remark}
	It is clear that if Lemma \ref{lem:7} holds true, then condition (A) holds true for $H(x,p,\lambda u)$ for $\lambda\in(0,\lambda_0)$. Then by Theorem \ref{ma}, one can deduce that item (3) in Theorem \ref{ma2}.
\end{remark}

\begin{proof}
	By \cite[Lemma 2.3]{WYZ0},  the sequence of $\{u^+_\lambda\} $ is uniformly bounded on $\lambda\in [0,1]$. Let
	$$
	M_1:=\max_{\lambda\in [0,1]} \|u^+_\lambda \|_\infty.
	$$
	
	It suffices to show that there exists  $\lambda_0>0$ such that for $\lambda\in (0,\lambda_0) $, there is no fixed points of contact Hamiltonian vector field generated by $H(x,p,\lambda u)$ on  
	\[
	\Omega_\lambda(M_1):= \big\{(x,p,u)\in T^* \ci\times [-M_1,M_1]:H(x,p,\lambda u)=0\big\}. 
	\]
%	This implies that 
%$$
%\frac{\partial G}{\partial p}(x,(u_{\lambda}^+)'(x),\lambda u^+_\lambda(x)) \neq 0, \quad  \forall x \in \mathbb{S}, \, \lambda \in (-\lambda_0,0) .
%$$

Assume by contradiction that there exist  $\{\lambda_n\}_{n\in \mathbb{N}}$ with 
	$
	\lambda_n\searrow0, 
	$ 
	and fixed points $(x_n,p_n,u_n)\in \Omega_{\lambda_n}(M_1)$ of the contact Hamiltonian vector fields. We have that
	$$
	\begin{cases}
		0= \dot{x}_n=\frac{\partial H }{\partial p}(x_n,p_n,\lambda_n u_n),\\
		%0=\dot p =-\frac{\partial H }{\partial x}(x,p,\lambda u)-\lambda
		% \frac{\partial H }{\partial u}(x,p,\lambda u) \cdot p, \\
		%0=\dot u=\frac{\partial H }{\partial p}(x,p,\lambda u) \cdot p-H(x,p,\lambda u),\\
		H(x_n,p_n,\lambda_n u_n)=0.
	\end{cases}
	$$
	In view of $H(x_n,p_n,\lambda_n u_n)=0$ and $p_n=\frac{\partial L}{\partial \dot{x}}(x_n,0,\lambda_n u_n)$, one obtains
	%\begin{align*}
	%	0=\dot p=&\,-H_x(x,p_\lambda(x,u),\lambda u)-\lambda H_u(x,p_\lambda(x,u),\lambda u) \cdot p_\lambda(x,u)\\
	%	=&\, -H_x(x,p_\lambda(x,u),0)-\lambda u \cdot \widehat H_{xu}  -\lambda H_u(x,p_\lambda(x,u),\lambda u) \cdot p_\lambda(x,u) \\
	%	=&\, -H_x(x,p_0(x),0)-\widehat H_{xp}\cdot (p_\lambda(x,u)-p_0(x) )) \\
	%	&\, \quad \quad \quad \quad \quad  -\lambda u \cdot \widehat H_{xu}  -\lambda H_u(x,p_\lambda(x,u),\lambda u) \cdot p_\lambda(x,u)\\
	%	=&\,-H_x(x,p_0(x),0)-f_1(\lambda ,x,u)  
	%\end{align*}
	%and
	\begin{align*}
		0=&\, H(x_n,\frac{\partial L}{\partial \dot{x}}(x_n,0,\lambda_n u_n),\lambda_n u_n)\\
		%=&H(x_0,\mathcal{L}^{-1}(x_0,0,\lambda u_0),0)+ \lambda u_0 \cdot \widehat H_u\\
		%=&\, H(x_0,\mathcal{L}^{-1}(x,0,0),0)+\widehat H_p\cdot \Big(\mathcal{L}^{-1}(x,0,\lambda u)+\mathcal{L}^{-1}(x,0,0) \Big)  +\lambda u \cdot \widehat H_u\\
		=&\, H(x_n,\frac{\partial L}{\partial \dot{x}}(x_n,0,0),0)+f(\lambda_n ,x_n,u_n),
	\end{align*}
	where 
	\begin{align*}
		&\,f(\lambda_n ,x_n,u_n):= \lambda_n u_n \cdot \int_0^1 \frac{\partial H}{\partial u} (x_n,\frac{\partial L}{\partial \dot{x}}(x_n,0,\lambda_n u_n),\tau \lambda_n u_n )\  d\tau \\
		&\,+\int_0^1 \frac{\partial H}{\partial p}\Big(x_n, \sigma \frac{\partial L}{\partial \dot{x}}(x_n,0,\lambda_n u_n)+(1-\sigma )\frac{\partial L}{\partial \dot{x}}(x_n,0,0), 0 \Big) \ d\sigma \cdot \Big(\frac{\partial L}{\partial \dot{x}}(x_n,0,\lambda_n u_n)
		-\frac{\partial L}{\partial \dot{x}}(x_n,0,0) \Big).  
	\end{align*}
Since $x_n\in\mathbb{S}$, we can take a convergent subsequence $\{x_{n_k}\}$ with $\lim_{k\to\infty}x_{n_k}=x^*$ for some $x^*\in\ci$.
Note that
	$$
	\lim_{ k \to +\infty} f(\lambda_{n_k},x_{n_k},u_{n_k})=0.
	$$
Hence
	\begin{align*}
	0=&\,\lim_{k\to +\infty} H(x_{n_k},\frac{\partial L}{\partial \dot{x}}(x_{n_k},0,\lambda_{n_k} u_{n_k}),\lambda_{n_k} u_{n_k}) \\
		=&\,\lim_{k\to +\infty} H(x_{n_k},\frac{\partial L}{\partial \dot{x}}(x_{n_k},0,0),0)  +\lim_{n\to +\infty}f(\lambda_{n_k},x_{n_k},u_{n_k} )\\
		=&\,H(x^*,\frac{\partial L}{\partial \dot{x}}(x^*,0,0),0)\\
		=&\, \min_{p\in \R} H(x^*,p,0),
	\end{align*}
	a contradiction to assumption (C).
\end{proof}

\subsection{Proof of Theorem \ref{ma2}}

\begin{proof}[Proof of Theorem \ref{ma2}]
Fix $\varphi\in C(\ci,\R)$ satisfying 
\begin{align}\label{9-350}
\min_{x\in\mathbb{S}}(\varphi(x)-u_+(x))=0<\max_{x\in\mathbb{S}}(\varphi(x)-u_+(x)). 
\end{align}
For any $\epsilon>0$, let $O_\epsilon$ denote the $\epsilon$-neighborhood of $I_\varphi:=\{x\in\ci: \varphi(x)=u_+(x)\}$. We assert that there exists $t_1>0$ such that 
	$$
	T_t \varphi(x)= \inf_{y\in O_\epsilon}h_{y,\varphi(y)}(x,t) \quad \forall (x,t) \in \mathbb{S} \times [t_1,+\infty). 
	$$
In fact, let $\sigma:=\min_{x\in \ci \backslash O_\epsilon}(\varphi(x)-u_+(x))$. Then by the definition of $I_\varphi$ and \eqref{9-350}, $\sigma>0$ is well defined. Let $u_\sigma:=u_++\sigma$. Then $\varphi(x)\geqslant u_\sigma (x)$, $\forall x\in \ci \backslash O_\epsilon$. By (D3), we have that 
	\begin{align}\label{8-8}
	\lim_{t\to+\infty}T_tu_\sigma(x)=+\infty,\quad \text{uniformly on}\ x\in \ci .
	\end{align}
	From (D1) there exist  $K>0$ independent of $\varphi$ and $T_\varphi>0$ such that 
	$$
	|T_t \varphi(x) | \leqslant   K , \quad \forall (x,t)\in \ci \times (T_\varphi,+\infty).
	$$ 
	From \eqref{8-8}, there is $t_1>T_\varphi$ such that
	\[
	T_tu_\sigma(x)\geqslant  K +1,\quad \forall (x,t)\in \ci \times (t_1,+\infty),
	\]
	where $t_1$ depends on $\epsilon$ and $\varphi$.
	Thus, for any $t\geqslant t_1$ and any $x\in \ci$, we get 
	\begin{align}\label{7-4}
		\inf_{y\in \ci \backslash O_\epsilon}h_{y,\varphi (y)}(x,t)\geqslant\, \inf_{y\in \ci \backslash O_\epsilon}h_{y,u_\sigma(y)}(x,t) 
		\geqslant \, \inf_{y\in  \ci}h_{y,u_\sigma(y)}(x,t)=T_tu_\sigma(x)\geqslant K +1.
	\end{align}
	 Hence, for any $t\geqslant t_1$, any $x\in \ci$, by \eqref{7-4} we have that
	\begin{align*}
		T_t\varphi (x)=\,
		\inf_{y\in \ci }h_{y,\varphi (y)}(x,t) 
		=\,\min\big\{\inf_{y\in O_\epsilon}h_{y,\varphi (y)}(x,t),\min_{y\in \ci \backslash O_\epsilon}h_{y,\varphi (y)}(x,t)\big\}
		=\inf_{y\in O_\epsilon}h_{y,\varphi (y)}(x,t).
	\end{align*}
	So far, we have shown the assertion.

	On one hand, since $I_\varphi \subset O_\epsilon$, we get
	$$
	T_t \varphi(x)= \inf_{y\in O_\epsilon}h_{y,\varphi(y)}(x,t)\leqslant \inf_{y\in I_\varphi }h_{y,\varphi(y)}(x,t)=\inf_{y\in I_\varphi }h_{y,u_+(y)}(x,t),\quad (x,t)\in\ci\times[t_1,+\infty).
	$$
	Thus, for any $(x,t) \in \mathbb{S} \times [0,+\infty)$, 
	$$
	T_{n\mathcal{T}+t} \varphi(x) \leqslant \inf_{y\in I_\varphi }h_{y,u_+(y)}(x,n\mathcal{T}+t), 
	$$
for any	$n\in \mathbb{N}$ with $n\geqslant \frac{t_1-t}{\mathcal{T}}$,
By \eqref{9-250} and Proposition \ref{pr-af} (4), we have 
  $h_{y,u_+(y)}(x,(n+1)\mathcal{T}+t)\leqslant h_{y,u_+(y)}(x,n\mathcal{T}+t)$. One can deduce that
	\begin{align}\label{7-20}
	\limsup_{n\to +\infty}  T_{n\mathcal{T}+t}\varphi(x) \leqslant \lim_{n\to +\infty} \inf_{y\in I_\varphi }h_{y,u_+(y)}(x,n\mathcal{T}+t)=:U_1(x,t). 
	\end{align}

	On the other hand, for $ n \geqslant  \frac{t_1-t}{\mathcal{T}} $ we get that
	$$
	T_{n\mathcal{T}+t} \varphi(x) \geqslant \inf_{y\in \overline{O}_\epsilon }h_{y,u_+(y)}(x,n\mathcal{T}+t),\quad \forall (x,t) \in \mathbb{S} \times [0,+\infty). 
	$$
And thus,  
\begin{align}\label{7-21}
	\liminf_{n\to +\infty}  T_{n\mathcal{T}+t}\varphi(x) \geqslant \lim_{n\to +\infty} \inf_{y\in \overline{O}_\epsilon }h_{y,u_+(y)}(x,n\mathcal{T}+t)=:U^\epsilon_2(x,t). 
\end{align}
By \eqref{7-20} and \eqref{7-21}, we have $U_1\geqslant U^\epsilon_2$.

Next, we prove that
	\begin{align}\label{7-26}
	\lim_{  \epsilon \to 0^+ } |U_1(x,t) - U^\epsilon_2(x,t)|=0, \quad\forall (x,t)\in\ci\times [0,+\infty). 
	\end{align}
	For any $x_1 \in I_\varphi$ and $x_2\in  \overline{O}_\epsilon$ with $|x_1-x_2|<\epsilon$, there exists $\tau_1$, $\tau_2\in [0,\mathcal{T}]$ such that $x(\tau_1)=x_1$ and $x(\tau_2)=x_2$, where $(x(t),\dot{x}(t),u_+(x(t)))$ is as in \eqref{9-199}.
	Without any loss of generality, assume $\tau_2>\tau_1$. Then
	$$
	h_{x(\tau_1),u_+(x(\tau_1))} (x(\tau_2) , \tau_2-\tau_1 )=u_+(x(\tau_2))
	$$
	which implies that
	\begin{align}\label{7-22}
		\begin{split}
	h_{x_1 ,u_+(x_1)}(x,n\mathcal{T}+t+\tau_2-\tau_1 )&=h_{x( \tau_1)  ,u_+(x( \tau_1))}(x,n\mathcal{T}+t+\tau_2-\tau_1)\\
		&=\inf_{y\in \ci} h_{y, h_{x( \tau_1),u_+(x( \tau_1))}(y,\tau_2-\tau_1 ) }(x, n\mathcal{T}+t ) \\
		&\leqslant h_{x( \tau_2)  ,u_+(x( \tau_2))}(x,n\mathcal{T}+t)\\
		&=h_{x_2 ,u_+(x_2)} (x,n\mathcal{T}+t).
		\end{split}
	\end{align}
	Let 
	$$
	v_i(x,t):= \lim_{n\to +\infty }  h_{x_i ,u_+(x_i)}(x,n\mathcal{T}+t), \quad i=1,2.
	$$
	Then by \eqref{7-22}
	$$
	v_1 (x,t+\tau_2-\tau_1 )\leqslant v_2 (x,t)  \quad \forall (x,t)\in \ci\times [0, +\infty).
	$$
	On the other hand, from $h_{x(\tau_2),u_+(x(\tau_2))} (x(\tau_1),\mathcal{T}+ \tau_1-\tau_2 )=u_+(x(\tau_1))$ and Proposition \ref{pr-af} (4), we have
	$$
	h_{x( \tau_2)  ,u_+(x( \tau_2))}(x,(n+1)\mathcal{T}+t) \leqslant h_{x_1,u_+(x_1)}(x,n\mathcal{T}+t+\tau_2-\tau_1 )
	$$
	which implies that 
	$$
	v_2 (x,t )\leqslant v_1 (x,t+\tau_2-\tau_1)  \quad \forall (x,t)\in \ci\times [0, +\infty).
	$$
	Therefore, we obtain that 
	$$v_2 (x,t )=v_1 (x,t+\tau_2-\tau_1), \quad 
	\forall (x,t)\in \ci\times [0, +\infty).
	$$
		Notice that
		\begin{align*}
			T_{\tau_2-\tau_1 } v_1(x,t)=&\, \inf_{y\in \ci} h_{y,v_1(y,t)}(x,\tau_2-\tau_1 )\\
			=&\, \inf_{y\in \ci} h_{y,\lim_{n\to \infty}h_{x_1,u_+(x_1)}(y,n\mathcal{T}+t) }(x,\tau_2-\tau_1 )  \\
			=&\, \lim_{n\to \infty}\inf_{y\in \ci} h_{y,h_{x_1,u_+(x_1)}(y,n\mathcal{T}+t) }(x,\tau_2-\tau_1 )\\
			=&\, \lim_{n\to \infty} h_{x_1,u_+(x_1)}\Big(x,n\mathcal{T}+t+\tau_2-\tau_1 \Big)\\
			=&\, v_1(x,t+\tau_2-\tau_1),
		\end{align*}
		which implies that	$v_1 (x,t+\tau_2-\tau_1) =T_{\tau_2-\tau_1}v_1(x,t)$.
	Since  $\tau_2-\tau_1\leqslant \frac{\epsilon}{\min_{s\in[0,T]}\|\dot x(s) \|} $, then 
	\begin{equation}\label{eq:5.5}
		\lim_{\epsilon\to 0^+} |v_2(x,t)-v_1(x,t) |=\lim_{\epsilon\to 0^+} |T_{\tau_2-\tau_1}v_1(x,t) -v_1(x,t) | =0.
	\end{equation}
%In fact, we have 
%\[
%	\lim_{\epsilon\to 0^+} |v_2(x,t)-v_1(x,t) |=0,
%\]
%uniformly on $(x,t)\in\ci\times[0,+\infty)$, since both $v_1$ and $v_2$ are time-periodic.

For any $x'_2\in \overline{O}_\epsilon$, there exists $x'_1\in I_\varphi$ such that \eqref{eq:5.5} holds. So, we have  that
	\begin{align*}\label{7-25}
		\begin{split}
		U^\epsilon_2(x,t)=&\, \inf_{y\in \overline{O}_\epsilon }   \lim_{n\to +\infty }  h_{y ,u_+(y)}(x,n\mathcal{T}+t) \\
		\geqslant &\,  \inf_{z \in I_\varphi }  \lim_{n\to +\infty }  h_{z ,u_+(z)}(x,n\mathcal{T}+t) - K(\epsilon)\\
		= &\, U^\epsilon_1(x,t) -K(\epsilon),
		\end{split}
	\end{align*}
	where $K(\epsilon)$ only depends on $\epsilon$ and $\lim_{\epsilon \to 0^+}K(\epsilon )=0$. Hence, \eqref{7-26} holds true.

	%Therefore, due to $\epsilon$ depends on $t$ and the Hausdorff distance of  $I_\varphi $ and $\overline{O}_\epsilon $ tends to 0, i.e.
	%$$
	%\lim_{t \to +\infty} d_H(I_\varphi, \overline{O}_\epsilon)=0.
	%$$ 
	Then we have $\liminf_{n\to +\infty}T_{n\mathcal{T}+t}\varphi(x)= \limsup_{n\to +\infty}T_{n\mathcal{T}+t}\varphi(x)$ and thus
	$$
	\lim_{n\to +\infty}T_{n\mathcal{T}+t}\varphi(x):=w_\varphi(x,t)
	$$
	exists. The proof of that $w_\varphi(x,t)$ is a nontrivial time periodic solution is quite similar to the one in the proof of Lemma \ref{lem00}. We omit it here for brevity.
	
\end{proof}

\section{Appendix}

Consider the contact Hamiltonian system
\begin{equation} \label{eq:ode}
	\left\{
	\begin{aligned}
		\dot x&=\frac{\partial H }{\partial p}(x,p,u),\\
		\dot p &=-\frac{\partial H }{\partial x}(x,p,u)-\frac{\partial H }{\partial u}(x,p,u) \cdot p, \quad (x,p,u)\in T^*M \times \R, \\ 
		\dot u&=\frac{\partial H }{\partial p}(x,p,u) \cdot p-H(x,p,u).
	\end{aligned}
	\right.
\end{equation}
Assume (H1)-(H3). Let us recall some known results on the Aubry-Mather-Ma\~n\'e-Fathi theory for \eqref{eq:ode} here. Most of the results in this section can be found in  \cite{WWY,WWY1,WWY2,WWY3}.

\medskip

\noindent $\bullet$ {\bf Variational principles}. 
First recall implicit variational principles for contact Hamiltonian system (\ref{eq:ode}), which connect contact Hamilton-Jacobi equations and contact Hamiltonian systems. 
\begin{proposition}\label{IVP}
	For any given $x_0\in M$, $u_0\in\mathbb{R}$, there exist two continuous functions $h_{x_0,u_0}(x,t)$ and $h^{x_0,u_0}(x,t)$ defined on $M\times(0,+\infty)$ satisfying	
	\begin{align}
		h_{x_0,u_0}(x,t)&=u_0+\inf_{\substack{\gamma(0)=x_0 \\  \gamma(t)=x} }\int_0^tL\big(\gamma(\tau),\dot{\gamma}(\tau),h_{x_0,u_0}(\gamma(\tau),\tau)\big)d\tau,\label{2-1}\\
		h^{x_0,u_0}(x,t)&=u_0-\inf_{\substack{\gamma(t)=x_0 \\  \gamma(0)=x } }\int_0^tL\big(\gamma(\tau),\dot{\gamma}(\tau),h^{x_0,u_0}(\gamma(\tau),t-\tau)\big)d\tau,\label{2-2}
	\end{align}
	where the infimums are taken among the Lipschitz continuous curves $\gamma:[0,t]\rightarrow M$.
	Moreover, the infimums in (\ref{2-1}) and \eqref{2-2} can be achieved. 
	If $\gamma_1$ and $\gamma_2$ are curves achieving the infimums \eqref{2-1} and \eqref{2-2} respectively, then $\gamma_1$ and $\gamma_2$ are of class $C^1$.
	Let 
	\begin{align*}
		x_1(s)&:=\gamma_1(s),\quad u_1(s):=h_{x_0,u_0}(\gamma_1(s),s),\,\,\,\qquad  p_1(s):=\frac{\partial L}{\partial \dot{x}}(\gamma_1(s),\dot{\gamma}_1(s),u_1(s)),\\
		x_2(s)&:=\gamma_2(s),\quad u_2(s):=h^{x_0,u_0}(\gamma_1(s),t-s),\quad   p_2(s):=\frac{\partial L}{\partial \dot{x}}(\gamma_2(s),\dot{\gamma}_2(s),u_2(s)).
	\end{align*}
	Then $(x_1(s),p_1(s),u_1(s))$ and $(x_2(s),p_2(s),u_2(s))$ satisfy equations \eqref{eq:ode} with 
	\begin{align*}
		x_1(0)=x_0, \quad x_1(t)=x, \quad \lim_{s\rightarrow 0^+}u_1(s)=u_0,\\
		x_2(0)=x, \quad x_2(t)=x_0, \quad \lim_{s\rightarrow t^-}u_2(s)=u_0.
	\end{align*}
\end{proposition}
We call $h_{x_0,u_0}(x,t)$ (resp. $h^{x_0,u_0}(x,t)$) a  forward (resp. backward) implicit action function associated with $L$
and the curves achieving the infimums in (\ref{2-1}) (resp. \eqref{2-2}) minimizers of $h_{x_0,u_0}(x,t)$ (resp. $h^{x_0,u_0}(x,t)$). The relation between forward and backward implicit action functions is as follows: for any given $x_0$, $x\in M$, $u_0$, $u\in\mathbb{R}$ and $t>0$,  
\begin{align*}\label{2-r}
	h_{x_0,u_0}(x,t)=u\quad  \text{if and only if}\quad  h^{x,u}(x_0,t)=u_0.
\end{align*}

%%%%%%%%%%%%%%%%%%%%%%%%%%%%%%%%%%%%%subsect 2.1
\medskip
%%%%%%%%%%%%%%%%%%%%%%%%%%%%%%%%%%%%%subsect 2.2

\noindent $\bullet$ {\bf Implicit action functions}. 
We now collect some basic properties of implicit action functions.

\begin{proposition}\label{pr-af} \ 
	\begin{itemize}
		\item [(1)] (Monotonicity).
		Given $x_{0} \in M, u_{0}, u_{1}, u_{2} \in \mathbb{R}$, Lagrangians $ L_{1}$ and $L_{2}$ satisfying (L1)-(L3),
		\begin{itemize}
			\item [(i)] if $u_{1}<u_{2}$, then $h_{x_{0}, u_{1}}(x, t)<h_{x_{0}, u_{2}}(x, t)$, for all $(x, t) \in M \times(0,+\infty)$;
			\item [(ii)] if $L_{1}<L_{2}$, then $h_{x_{0}, u_{0}}^{L_{1}}(x, t)<h_{x_{0}, u_{0}}^{L_{2}}(x, t)$, for all $(x, t) \in M \times(0,+\infty)$
			where $h_{x_{0}, u_{0}}^{L_{i}}(x, t)$ denotes the forward implicit action function associated with $L_{i}, i=1,2 .$
		\end{itemize}
		\item [(2)] (Lipschitz continuity).
		The function $(x_0,u_0,x,t)\mapsto h_{x_0,u_0}(x,t)$ is  Lipschitz  continuous on $M\times[a,b]\times M\times[c,d]$ for all real numbers $a$, $b$, $c$, $d$
		with $a<b$ and $0<c<d$.
		\item [(3)] (Minimality).
		Given $x_0$, $x\in M$, $u_0\in\mathbb{R}$ and $t>0$, let
		$S^{x,t}_{x_0,u_0}$ be the set of the solutions $(x(s),p(s),u(s))$ of (\ref{eq:ode}) on $[0,t]$ with $x(0)=x_0$, $x(t)=x$, $u(0)=u_0$.
		Then
		\[
		h_{x_0,u_0}(x,t)=\inf\{u(t): (x(s),p(s),u(s))\in S^{x,t}_{x_0,u_0}\}, \quad \forall (x,t)\in M\times(0,+\infty).
		\]
		\item [(4)] (Markov property).
		Given $x_0\in M$, $u_0\in\mathbb{R}$, 
		\[
		h_{x_0,u_0}(x,t+s)=\inf_{y\in M}h_{y,h_{x_0,u_0}(y,t)}(x,s)
		\]
		for all  $s$, $t>0$ and all $x\in M$. Moreover, the infimum is attained at $y$ if and only if there exists a minimizer $\gamma$ of $h_{x_0,u_0}(x,t+s)$ with $\gamma(t)=y$.
		\item [(5)] (Reversibility).
		Given $x_0$, $x\in M$ and $t>0$, for each $u\in \mathbb{R}$, there exists a unique $u_0\in \mathbb{R}$ such that
		\[
		h_{x_0,u_0}(x,t)=u.
		\]
	\end{itemize}
\end{proposition}

\medskip

\begin{proposition}\label{pr-af1} \ 
	\begin{itemize}
		\item [(1)]({\it Monotonicity}).
		Given $x_{0} \in M$ and $u_{1}, u_{2} \in \mathbb{R}$, Lagrangians $L_{1}, L_{2}$ satisfying (L1)-(L3),
		\begin{itemize}
			\item [(i)] if $u_{1}<u_{2}$, then $h^{x_{0}, u_{1}}(x, t)<h^{x_{0}, u_{2}}(x, t)$, for all $(x, t) \in M \times(0,+\infty) ;$
			\item [(ii)] if $L_{1}>L_{2}$, then $h_{L_{1}}^{x_{0}, u_{0}}(x, t)<h_{L_{2}}^{x_{0}, u_{0}}(x, t)$, for all $(x, t) \in M \times(0,+\infty)$, where
			$h_{L_{i}}^{x_{0}, u_{0}}(x, t)$ denotes the backward implicit action function associated with $L_{i}$, $i =1,2$.
		\end{itemize}
		\item [(2)] (Lipschitz continuity).
		The function $(x_0,u_0,x,t)\mapsto h^{x_0,u_0}(x,t)$ is  Lipschitz  continuous on $M\times[a,b]\times M\times[c,d]$ for all real numbers $a$, $b$, $c$, $d$
		with $a<b$ and $0<c<d$.
		
		\item [(3)] (Maximality).
		Given $x_0$, $x\in M$, $u_0\in\mathbb{R}$ and $t>0$, let
		$S_{x,t}^{x_0,u_0}$ be the set of the solutions $(x(s),p(s),u(s))$ of  
		(\ref{eq:ode}) on $[0,t]$ with $x(0)=x$, $x(t)=x_0$, $u(t)=u_0$.
		Then
		\[
		h^{x_0,u_0}(x,t)=\sup\{u(0): (x(s),p(s),u(s))\in S_{x,t}^{x_0,u_0}\}, \quad \forall (x,t)\in M\times(0,+\infty).
		\]
		\item [(4)] (Markov property).
		Given $x_0\in M$, $u_0\in\mathbb{R}$, 
		\[
		h^{x_0,u_0}(x,t+s)=\sup_{y\in M}h^{y,h^{x_0,u_0}(y,t)}(x,s)
		\]
		for all  $s$, $t>0$ and all $x\in M$. Moreover, the supremum is attained at $y$ if and only if there exists a minimizer $\gamma$ of $h^{x_0,u_0}(x,t+s)$, such that $\gamma(t)=y$.
		\item [(5)] (Reversibility).
		Given $x_0$, $x\in M$, and $t>0$, for each $u\in \mathbb{R}$, there exists a unique $u_0\in \mathbb{R}$ such that
		\[
		h^{x_0,u_0}(x,t)=u.
		\]
	\end{itemize}
\end{proposition}

%============================================Subsect. 2.2

\medskip

%============================================Subsect. 2.3
\noindent $\bullet$ {\bf Solution semigroups}.

Let us recall two  semigroups of operators introduced in \cite{WWY1}.  Define a family of nonlinear operators $\{T^-_t\}_{t\geqslant 0}$ from $C(M,\mathbb{R})$ to itself as follows. For each $\varphi\in C(M,\mathbb{R})$, denote by $(x,t)\mapsto T^-_t\varphi(x)$ the unique continuous function on $ (x,t)\in M\times[0,+\infty)$ such that
\begin{equation*}\label{eq:semigroup}
	T^-_t\varphi(x)=\inf_{\gamma}\left\{\varphi(\gamma(0))+\int_0^tL(\gamma(\tau),\dot{\gamma}(\tau),T^-_\tau\varphi(\gamma(\tau)))d\tau\right\},
\end{equation*}
where the infimum is taken among the absolutely continuous curves $\gamma:[0,t]\to M$ with $\gamma(t)=x$.  It was also proved  in \cite{WWY1} that $\{T^-_t\}_{t\geqslant 0}$ is a semigroup of operators and the function $(x,t)\mapsto T^-_t\varphi(x)$ is a viscosity solution of $\partial_tw+H(x,\partial_xw,w)=0$ with initial condition $w(x,0)=\varphi(x)$. Thus, we call $\{T^-_t\}_{t\geqslant 0}$ the backward solution semigroup.

Similarly, one can define another semigroup of operators $\{T^+_t\}_{t\geq 0}$, called the forward solution semigroup 
by
\begin{equation*}\label{2-4}
	T^+_t\varphi(x)=\sup_{\gamma}\left\{\varphi(\gamma(t))-\int_0^tL(\gamma(\tau),\dot{\gamma}(\tau),T^+_{t-\tau}\varphi(\gamma(\tau)))d\tau\right\},
\end{equation*}
where the supremum is taken among curves $\gamma\in C^{ac}([0,t],M)$ with $\gamma(0)=x$.

We collect some basic properties of the solution semigroups. 

\begin{proposition}\label{pr-sg}
	Let $\varphi$, $\psi\in C(M,\R)$. 
	\begin{itemize}
		\item [(1)](Monotonicity). If $\psi<\varphi$, then $T^{\pm}_t\psi< T^{\pm}_t\varphi$, $\forall t\geqslant 0$.
		\item [(2)](Local Lipschitz continuity). The function $(x,t)\mapsto T^{\pm}_t\varphi(x)$ is locally Lipschitz on $M\times (0,+\infty)$.
		\item[(3)]($e^{\kappa t}$-expansiveness). $\|T^{\pm}_t\varphi-T^{\pm}_t\psi\|_\infty\leqslant e^{\kappa t}\cdot\|\varphi-\psi\|_\infty$,  $\forall t\geqslant 0$.
		\item[(4)] (Continuity at the origin). $\lim_{t\rightarrow0^+}T^{\pm}_t\varphi=\varphi$.
		\item[(5)] (Representation formula). For each  $\varphi\in C(M,\R)$,	
		\begin{itemize}
			\item [(i)]	$T^-_t\varphi(x)=\inf_{y\in M}h_{y,\varphi(y)}(x,t)$,\quad  $\forall (x,t)\in M\times(0,+\infty)$;
			\item [(ii)] $T^+_t\varphi(x)=\sup_{y\in M}h^{y,\varphi(y)}(x,t)$,\quad  $\forall (x,t)\in M\times(0,+\infty)$.
		\end{itemize}
		
		\item[(6)] (Semigroup). $\{T^{\pm}_t\}_{t\geqslant 0}$ are one-parameter semigroup of operators. For all $x_0$, $x\in M$, all $u_0\in\mathbb{R}$ and all $s$, $t>0$, 	\begin{itemize}
			\item [(i)] $
			T^-_sh_{x_0,u_0}(x,t)=h_{x_0,u_0}(x,t+s)$,\quad  $T^-_{t+s}\varphi(x)=\inf_{y\in M}h_{y,T^-_s\varphi(y)}(x,t)$;
			\item [(ii)] $
			T^+_sh^{x_0,u_0}(x,t)=h^{x_0,u_0}(x,t+s)$, \quad  $T^+_{t+s}\varphi(x)=\sup_{y\in M}h^{y,T^+_s\varphi(y)}(x,t)$.
		\end{itemize}
	\end{itemize}
\end{proposition}

%%%%%%%%%%%%%%%%%%%%%%%%%%%%%%%%%%%%%subsect 2.3

\medskip

%%%%%%%%%%%%%%%%%%%%%%%%%%%%%%%%%%%%%subsect 2.4
\noindent $\bullet$ {\bf Weak KAM solutions}. 
Following Fathi (see, for instance, \cite{Fat-b}), one can define weak KAM solutions of  
\begin{align}\label{wkam}
	H(x,Du(x),u(x))=0.	
\end{align}

\begin{definition}
	\label{bwkam}
	A function $u\in C(M,\R)$ is called a backward weak KAM solution of \eqref{wkam} if
	\begin{itemize}
		\item [(1)] for each continuous piecewise $C^1$ curve $\gamma:[t_1,t_2]\rightarrow M$, we have
		\begin{align}\label{do}
			u(\gamma(t_2))-u(\gamma(t_1))\leqslant\int_{t_1}^{t_2}L(\gamma(s),\dot{\gamma}(s),u(\gamma(s)))ds;
		\end{align}
		\item [(2)] for each $x\in M$, there exists a $C^1$ curve $\gamma:(-\infty,0]\rightarrow M$ with $\gamma(0)=x$ such that
		\begin{align}\label{cali1}
			u(x)-u(\gamma(t))=\int^{0}_{t}L(\gamma(s),\dot{\gamma}(s),u(\gamma(s)))ds, \quad \forall t<0.
		\end{align}
	\end{itemize}
	
	Similarly, 	a function $v\in C(M,\R)$ is called a forward weak KAM solution of  \eqref{wkam} if it satisfies (1) and  
	for each $x\in M$, there exists a $C^1$ curve $\gamma:[0,+\infty)\rightarrow M$ with $\gamma(0)=x$ such that
	\begin{align}\label{cali2}
		v(\gamma(t))-v(x)=\int_{0}^{t}L(\gamma(s),\dot{\gamma}(s),v(\gamma(s)))ds,\quad \forall t>0.
	\end{align}
	We say that $u$ in \eqref{do} is a dominated function by $L$. We call curves satisfying \eqref{cali1} (resp. \eqref{cali2}),  $(u,L,0)$-calibrated curves (resp. $(v,L,0)$-calibrated curves). We use $\mathcal{S}_-$ (resp. $\mathcal{S}_+$) to denote the set of all backward (resp. forward) weak KAM solutions.
\end{definition}

\begin{proposition}\label{pr-fix}\ 
	\begin{itemize}
		\item
		[(1)] $u\in\mathcal{S}_-$ if and only if $T^-_tu=u$ for all $t\geqslant 0$.
		\item [(2)] $v\in\mathcal{S}_+$ if and only if $T^+_tv=v$ for all $t\geqslant 0$.
	\end{itemize}
\end{proposition}

\medskip

Let $\Phi_t^H$ denote the local flow of \eqref{eq:ode}.
\begin{definition}
	\begin{itemize}
		\item[(1)] (Globally minimizing orbits)
		A curve $(x(\cdot),u(\cdot)):\R \to M \times \R $ is called globally minimizing , if it is locally Lipschitz and for each $t_1<t_2\in\R$, there holds 
		$$
		u(t_2)=h_{x(t_1),u(t_1)}(x(t_2),t_2-t_1 ).
		$$
		\item[(2)](Static curves) A curve $(x(\cdot),u(\cdot)):\R \to M \times \R $ is called static, if it is globally minimizing and for each $t_1,t_2\in \R$, there holds
		$$
		u(t_2)=\inf_{s>0} h_{x(t_1),u(t_1)}(x(t_2),s).
		$$
		If a curve $(x(\cdot),u(\cdot)):M \times \R$ is static, then $(x(t),p(t),u(t))$ with $t\in \R$  is an orbit of $\Phi_t^H$, where $p(t)=\frac{\partial L}{\partial \dot{x}}(x(t),\dot x(t),u(t )) $. We call it a static orbit of $\Phi_t^H$.
	\end{itemize}
\end{definition}

\begin{definition}[Aubry set]\label{audeine}
	We call the set of all static orbits Aubry set of $H$, denoted by $\tilde{\mathcal{A}}$. We call $\mathcal{A}:=\pi\tilde{\mathcal{A}}$ the projected Aubry set, where $\pi:T^*M\times\R\rightarrow M$ denotes the canonical projection.
\end{definition}

We also call $\bar{\mathcal{L}}(\tilde{\mathcal{A}})\subset TM\times\R$ the Aubry set.

\section{Acknowledgements} 
Kaizhi Wang is supported by NSFC Grant No. 12171315, 11931016.
Jun Yan is supported by NSFC Grant No.  12171096, 11790273.


\medskip

\begin{thebibliography}{99}\small
	\addcontentsline{toc}{section}{References}
	\renewcommand{\baselinestretch}{0.0}
	\setlength\itemsep{-2pt}
	
	
\bibitem{Bar} M. Bardi and I. Capuzzo-Dolcetta, Optimal control and viscosity solutions of Hamilton-Jacobi-Bellman equations. Systems \& Control: Foundations \& Applications. Birkh\"auser, 1997.


\bibitem{BS}
G. Barles and P. E. Souganidis, Some counterexamples on the
asymptotic behavior of the solutions of Hamilton-Jacobi equations, C. R. Acad. Sci. Paris S\'er. I
Math. \textbf{330} (2000), 963--968. 


\bibitem{BR}
P. Bernard, J. Roquejoffre, Convergence to time-periodic solutions in time-periodic Hamilton-Jacobi equations on the circle,
Comm. Partial Differential Equations \textbf{29} (2004),  457--469.



\bibitem{CL} M. Crandall and P.-L. Lions, Viscosity solutions of Hamilton-Jacobi equations, Trans. Amer. Math. Soc. \textbf{277} (1983), 1--42.
	
	
	
%\bibitem{CIP}G. Contreras, R. Iturriaga, G. Paternain and M. Paternain, Lagrangian graphs, minimizing measures and Ma\~n\'e's critical values, Geom. Funct. Anal. \textbf{8} (1998),  788--809.
	
	
%\bibitem{DFIZ}
%A. Davini, A. Fathi, R. Iturriaga and M. Zavidovique, Convergence of the solutions of the discounted Hamilton-Jacobi equation: convergence of the discounted solutions, Invent. Math. \textbf{206} (2016), 29--55.
	
\bibitem{FM}
A. Fathi, J. Mather, Failure of convergence of the Lax-Oleinik semi-group in the time-periodic case,  Bull. Soc. Math. France \textbf{128} (2000), 473--483.



\bibitem{Fat-b}
A. Fathi, Weak KAM Theorem and Lagrangian Dynamics. 
\url{http://www.math.u-bordeaux.fr/~pthieull/Recherche/KamFaible/Publications/Fathi2008_01.pdf}
	
	
\bibitem{I}	
H. Ishii, Asymptotic solutions for large time of Hamilton-Jacobi equations, in: International Congress of Mathematicians, Vol. III, Eur. Math. Soc., Zurich, 2006,  213--227.


\bibitem{L}
P.-L. Lions,  Generalized solutions of Hamilton-Jacobi equations. Pitman, Boston, 1982.



\bibitem{LPV}
P.-L. Lions, G. Papanicolaou and S. R. S. Varadhan, Homogenization of Hamilton-Jacobi Equations, unpublished.

\bibitem{Ma}
R. Ma\~n\'e,  Generic properties and problems of minimizing measures of Lagrangian systems, Nonlinearity \textbf{9} (1996),  273--310.	


\bibitem{T}
H. Tran, Hamilton-Jacobi equations: viscosity and applications, 
\url{https://people.math.wisc.edu/~hung/HJ%20equations-viscosity%20solutions%20and%20applications-v2.pdf}


\bibitem{WY}
K. Wang, J. Yan, A new kind of Lax-Oleinik type operator with parameters for time-periodic positive definite Lagrangian systems, Commun. Math. Phys.  \textbf{309} (2012), 663--691.



\bibitem{WY1}
K. Wang, J. Yan, The rate of convergence of the new Lax-Oleinik type operator for time-periodic positive definite Lagrangian systems, Nonlinearity \textbf{25} (2012), 2039--2057.

	
\bibitem{WWY} K. Wang, L. Wang and J. Yan, Implicit variational principle for contact Hamiltonian systems, Nonlinearity \textbf{30} (2017), 492--515.
	
	
	
\bibitem{WWY1} K. Wang, L. Wang and J. Yan,  Variational principle for contact Hamiltonian systems and its applications, J. Math. Pures Appl. \textbf{123} (2019), 167--200.
	
	
	
\bibitem{WWY2} K. Wang, L. Wang and J. Yan, Aubry-Mather theory for contact Hamiltonian systems, Commun. Math. Phys. \textbf{366} (2019), 981--1023.
	
	
	
\bibitem{WWY3}
K. Wang, L. Wang and J. Yan, Weak KAM solutions of Hamilton-Jacobi equations with decreasing dependence on unknown functions, J. Differential Equations \textbf{286} (2021), 411--432.
	
	
	
  \bibitem{WY2021}
   K. Wang, J. Yan,  Viscosity solutions of contact Hamilton-Jacobi equations without monotonicity assumptions, arXiv: 2107.11554.
%    
%    \bibitem{WYZ1}
%K. Wang, J. Yan and K. Zhao, Finite-time convergence of solutions of Hamilton-Jacobi equations,  Proceedings of the American Mathematical Society.
	
	
	
\bibitem{WYZ0}
Y. Wang, J. Yan and J. Zhang, Convergence of viscosity solutions of generalized contact Hamilton-Jacobi equations, Arch. Rational Mech. Anal. \textbf{241} (2021) 885--902.
\end{thebibliography}
\end{document}